\documentclass[11pt]{llncs}
\usepackage[letterpaper,hmargin=1.1in,vmargin=1.15in]{geometry}
\usepackage{amsmath,amssymb,amsbsy,amsfonts,latexsym,amsopn,amstext,amsxtra,euscript,amscd,color}
\usepackage{array}
\usepackage{booktabs}
\usepackage{graphicx}
\usepackage{slashbox}
\usepackage{psfrag}
\usepackage{epsfig}
\usepackage[sort&compress,numbers,square]{}

\def\00{{\bf 0}}
\def\+{\oplus}

\def\\{\cr}
\def\({\left(}
\def\){\right)}

\newcommand{\BBR}{\mathbb{R}}
\newcommand{\BBN}{\mathbb{N}}
\newcommand{\BBC}{\mathbb{C}}

\providecommand{\newoperator}[3]{%
  \newcommand*{#1}{\mathop{#2}#3}}

\newoperator{\FD}{\mathrm{FD}}{\nolimits}

\begin{document}
\title{ Homotopy perturbation transform method for solving fractional partial differential equations with proportional delay}
\author{Brajesh Kumar Singh\thanks{Address for Correspondence: Department of Applied Mathematics, Babasaheb Bhimrao Ambedkar University Lucknow-226 025 (UP) INDIA}, Pramod Kumar}
\institute{Department of Applied Mathematics, School for Physical Sciences, \\ Babasaheb Bhimrao Ambedkar University Lucknow-226 025 (UP)  INDIA \\
\email{bksingh0584@gmail.com, bbaupramod@gmail.com}}
\date{\today}

\maketitle
\begin{abstract}
This paper deals the implementation of \emph{homotopy perturbation transform method} (HPTM) for numerical computation of initial valued autonomous system of time-fractional partial differential equations (TFPDEs) with proportional delay, including generalized Burgers equations with proportional delay. The HPTM is a hybrid of Laplace transform and homotopy perturbation method. To confirm the efficiency and validity of the method, the computation of three test problems of TFPDEs with proportional delay presented. The proposed solutions are obtained in series form, converges very fast. The scheme seems very reliable, effective and efficient powerful technique for solving various type of physical models arising in sciences and engineering.
\end{abstract}

\noindent {\bf Keywords:} Homotopy perturbation transform method; fractional derivative, in Caputo sense; autonomous differential equations; fractional partial deferential equation with proportional delay; generalized Burgers equations with proportional delay

\noindent \textbf{Mathematics Subject Classification (2010):}{  35R11 $\cdot$ 49Mxx $\cdot$ 65H20 $\cdot$ 60Hxx}

\section{Introduction}
Fractional differential equation have achieved great attention among researchers due to its wide range of applications in
various meaningful phenomena in fluid mechanics, electrical networks, signal processing, diffusion, reaction processes and other fields of science and engineering \cite{MR93,Kilbas06,CM71}, among them, the non-linear oscillation of earthquake can be modeled with fractional derivatives  \cite{He98a}, the fluid-dynamic traffic model having fractional derivatives \cite{He99a} can eliminate the deficiency arising from the assumption of continuum traffic flow, fractional non linear complex model for seepage flow in porous media in \cite{He98}.

Indeed, it is too tough to find an exact solution of a wide class of the differential equation. Keeping all this in mind, various type of vigorous techniques has been developed to find an approximate solution of such type of fractional differential equations, among others,  generalized differential transform method \cite{LH11},  variational iteration method \cite{SE15}, local fractional variational iteration method \cite{YBKM14}, reproducing kernel Hilbert space method \cite{GC12}, Adomian decomposition method \cite{MO06} and homotopy analysis method \cite{SE13}, Fractional reduced differential transform method \cite{SAS14a,SKAS14,SS15,SM16}.

In the recent, vigourous techniques with Laplace transform has been developed, among them, see \cite{KKAR14,GKO13,KH11,KG12,KGB13,KGK12,KGHV12,KKA13}. Among others, HPTM has been employed for solving fractional model of Navier-Stokes equation \cite{KSK15}, optimal control problems \cite{GR16}, fractional coupled sine-Gordon equations \cite{SSahoo15}, Falkner-Skan wedge flow \cite{MAAR16}, time- and space-fractional coupled Burgers' equations \cite{SKS16}, strongly nonlinear oscillators \cite{MEA09}, nonlinear boundary values problems \cite{NINO16,KW11}, non-homogeneous
partial differential equations with a variable coefficient \cite{MFKY11}. The reader also refer to read \cite{SUE16}.

The partial functional differential equations with proportional delays, a special class of delay partial differential equation, arises specially in the field of biology, medicine, population ecology, control systems and climate models \cite{Wu96}, and complex economic macrodynamics \cite{Keller10}.

In this paper, we obtain the numerical solution of initial valued autonomous system of TFPDEs with proportional delay \cite{SUE16} defined by
\begin{equation*}
\left\{ \begin{array}{ll}
\mathcal{D}_t^{\alpha} \(u(x, t)\) =  f\left(x, u(a_0 x, b_0 t), \frac{\partial }{\partial x}u(a_1 x, b_1 t), \ldots, \frac{\partial^m }{\partial x^m}u(a_m x, b_m t) \right), \\
u^{k}(x, 0)=\psi_k(x).
\end{array}  \right.
\end{equation*}
where $a_i, b_i \in (0, 1)$ for all $i\in N \cup \{0\}$. $g_k$ is initial value and $f$ is the differential operator, the independent variables $(x,t)$ (where $t$ denote time, and $x$-space variable) generally denotes the position in space or size of cells, maturation level, at a time while its solution may be the voltage, temperature, densities of different particles, form instance, chemicals, cells, etc.. One significant example of the model: Korteweg-de Vries (KdV) equation, arise in the research of shallow water waves is as follow:
\begin{equation*}
\mathcal{D}_t^{\alpha} \(u(x, t)\) =  b u \frac{\partial }{\partial x} u(a_0x, b_0 t)+ \frac{\partial^3 }{\partial x^3} u(a_1 x, b_1 t), \quad 0< \alpha < 1.
\end{equation*}
where $b$ is a constant. The another well-known model: time-fractional nonlinear Klein--Gordon equation with proportional delay, aries in quantum field theory to describe nonlinear wave interaction
\begin{equation*}
\mathcal{D}_t^{\alpha} \(u(x, t)\) =  u \frac{\partial^2 }{\partial x^2} u(a_0x, b_0 t)- b u(a_1 x, b_1 t)-F(u(a_2 x, b_2 t))+ h(x, t), \quad 1< \alpha < 2.
\end{equation*}
where $b$ is a constant, $h(x, t)$ known analytic function and $F$ is the nonlinear operator of $u(x, t)$.  For details of various type of models, we refer the reader to \cite{SUE16,Wu96} and the references therein.

To best of my knowledge, a little literature of numerical techniques to solve TFPDE with delay, among them, Chebyshev pseudospectral method for linear differential and differential-functional parabolic equations by Zubik-Kowal \cite{ZK2000}, spectral collocation $\&$ waveform relaxation methods by Zubik-Kowal and Jackiewicz \cite{ZJ06} and iterated pseudospectral method \cite{MZ05} for nonlinear delay partial differential equations,  two dimensional differential transform method (2D-DTM) and RDTM for partial differential equations with proportional delay by Abazari and Ganji \cite{AG11}, Abazari and Kilicman \cite{AK11} used DTM for nonlinear integro-differential equations with proportional delay, group analysis method for nonhomogeneous  mucilaginous Burgers equation with proportional delay due to Tanthanuch \cite{T12}, homotopy perturbation method for TFPDE with proportional delay by Sakar et al.\cite{SUE16} and Shakeri-Dehghan \cite{SD08}, and Biazar ad Ghanbari \cite{BG12}, variational iteration method (VIM) for solving a neutral functional-differential equation with proportional delays by Chena and Wang \cite{CW10}, functional constraints method for the exact solutions of nonlinear delay reaction-diffusion equations by Polyanin and Zhurov \cite{PZ14}, and so on.

In this paper, our main goal is to propose an alternative numerical solution of initial valued autonomous system of time fractional partial differential equation with proportional delay \cite{SUE16}. The paper is organized into five more sections, which follow this introduction. Specifically, Section \ref{sec-fracks} deals with revisit of fractional calculus.  Section \ref{sec-impliment} is devoted to the procedure for the implementation of  the HPTM for the problem \eqref{eqn-PD}.  Section \ref{sec-numeric-stdy} is concerned with three test problems with the main aim to establish the convergency and effectiveness of the HPTM.  Finally, Section \ref{sec-conclu} concludes the paper with reference to
critical analysis and research perspectives.

\section{Preliminaries} \label{sec-fracks}
This section revisit some basic definitions of definitions of fractional calculus due to Liouville \cite{MR93}, which we need to complete the paper
\begin{definition}
Let $ \mu  \in \BBR$ and $m \in \BBN$. A function $f: \BBR^{+} \to \BBR$ belongs to $\BBC_{\mu}$ if there exists $k \in \BBR$, $k > \mu$ and $g \in C[0, \infty)$ such that $f(x)=x^k g(x)$, $ \forall x \in \BBR^+$. Moreover,  $f\in\BBC_{\mu}^m$  if
 $f^{(m)}\in \BBC_{\mu}$.
\end{definition}
\begin{definition} Let $\mathcal{J}_t^\alpha~ (\alpha \geq 0)$ be Riemann-Liouville fractional integral operator and let $f \in \BBC_{\mu}$, then
\begin{flushleft}
\begin{description}
  \item[(*)]  $ \mathcal{\mathcal{J}}_t^{\alpha}  f\left(t \right){\rm{ = }}\frac{{\rm{1}}}{{\Gamma \left( \alpha  \right)}}\int\limits_{\rm{0}}^{\rm{t}} {\left( {{\rm{t - \tau}}} \right)^{\alpha  - 1} } f\left( \tau \right){\rm{d\tau,}} ~~  \mbox { if } \alpha >0, $
  \item[(**)]$\mathcal{J}_t^0  f\left(t \right){\rm{ = }} f(t)$, \quad where $\Gamma \left( z \right) := \int\limits_0^\infty  {e^{ - t} t^{z - 1} dt,z \in \BBC} $.
\end{description}
\end{flushleft}
\end{definition}
For $f \in \BBC_{\mu}, \mu \geq -1, \alpha, \beta \geq 0 $ and $\gamma >-1$, the operator $\mathcal{J}_t^{\alpha}$ satisfy the following properties
\begin{itemize}
  \item [i)]$\mathcal{J}_t^{\alpha} \mathcal{J}_t^{\beta} f(x) = \mathcal{J}_t^{\alpha+\beta} f(x)=J_t^{\beta} \mathcal{J}_t^{\alpha} f(x),$ \quad $\textbf{ ii)}$ $\mathcal{J}_t^{\alpha} x^{\gamma} = \frac{\Gamma {(1+\gamma)}}{\Gamma {(1+\gamma+\alpha)}}x^{\alpha+\gamma}.$
\end{itemize}
The Caputo fractional differentiation operator $\mathcal{D}^\alpha_{t}$ defined as follows:
\begin{definition} Let $f \in \BBC_{\mu}$, $\mu \geq -1$ and $m - 1 < \alpha  \le m, {\rm{ }}m \in \BBN$. Then
\begin{equation}\label{eqn-d2}
\mathcal{D}_t^\alpha  f\left( t \right){\rm{ = }}\mathcal{J}_t^{m - \alpha } \mathcal{D}_t^m f\left( t \right) = \frac{{\rm{1}}}{{\Gamma \left( {m - \alpha } \right)}}\int\limits_{\rm{0}}^{\rm{t}} {\left( {{\rm{t - \tau}}} \right)^{m - \alpha  - 1} } f^{\left( m \right)} \left( \tau \right){\rm{d\tau,}}
\end{equation}
\end{definition}
Moreover, the operator $\mathcal{D}_t^\alpha$ satisfy following the basic properties
\begin{lemma} \label{lem-caput} Let $m -1 < \alpha \le m,{\rm{ }}m \in \BBN,$ and $f \in \BBC_\mu ^m ,{\rm{ }}\mu \ge {\rm{ - 1}}$, and $\gamma > \alpha -1$, then
\begin{flushleft}
$\begin{array}{ll}
(a) \mathcal{D}_t^{\alpha}\mathcal{D}_t^{\beta} f(x) = \mathcal{D}_t^{\alpha+\beta} f(x); \qquad \quad (b)~~\mathcal{D}_t^{\alpha} x^{\gamma} = \frac{\Gamma {(1+\gamma)}}{\Gamma {(1+\gamma-\alpha)}}x^{\gamma-\alpha},\\
(c)~ \mathcal{D}_t^\alpha  \mathcal{J}_t^\alpha  f\left( t \right){\rm{ = }}f\left( t \right), (d) ~ \mathcal{J}_t^\alpha  \mathcal{D}_t^\alpha  f\left( t \right){\rm{ = }}f\left( t \right) - \sum\limits_{k = 0}^m {f^{\left( k \right)} \left( {0^ +  } \right)\frac{{t^k }}{{k!}}} ,& \mbox{ for }~~ {\rm{ t > 0.}} \\
 \end{array}$ 
\end{flushleft}
\end{lemma}
For more details on fractional derivatives, one can refer \cite{MR93,Kilbas06,CM71}.

\begin{definition} The Laplace transform of a piecewise continuous function $u(t)$ in $(0, \infty)$ is defined by
\begin{equation}\label{eqn-d03}
\mathcal{U}(s)=\mathcal{L}\{u(t)\}=\int_{0}^{\infty} u(t)\exp(-st)dt,
\end{equation}
where $s$ is a parameter. Moreover, for the Caputo derivative $\mathcal{D}_t^\alpha u\left( t \right)$ and Riemann--
Liouville fractional integral $ \mathcal{J}_t^\alpha u\left( t \right)$ of a function $u \in \BBC_{\mu} ~~(\mu\geq-1)$, the laplace transform \cite{MR93,Kilbas06} is defined as
\begin{equation}\label{eqn-d03}
\begin{split}
& \mathcal{L}\{\mathcal{J}_t^\alpha u\left(t\right)\}=s^{- \alpha} \mathcal{U}(s),\\
& \mathcal{L}\{\mathcal{D}_t^\alpha f\left(t\right)\}=s^{\alpha} \mathcal{U}(s) -\sum_{r=0}^{m-1}s^{\alpha-r-1}u^{(r)}(0+), \quad (m-1 < \alpha \leq m)\\
\end{split}
\end{equation}
\end{definition}
\section{Implementation: HPTM for TFPDEs with proportional delay }\label{sec-impliment}
\noindent This section describes the implementation of HPTM to the initial valued autonomous system of
TFPDEs with proportional delay, defined as below:
\begin{equation}\label{eqn-PD}
\left\{ \begin{array}{ll}
\mathcal{D}_t^{\alpha} \(u(x, t)\) =  f\left(x, u(a_0 x, b_0 t), \frac{\partial }{\partial x}u(a_1 x, b_1 t), \ldots, \frac{\partial^m }{\partial x^m}u(a_m x, b_m t) \right), \\
u(x, 0)=\psi(x).
\end{array}  \right.
\end{equation}
Taking Laplace transform on Eq. \eqref{eqn-PD}, we get
\begin{equation}\label{eqn-RDT-SG-LT}
 \mathcal{U}(x, s) = \frac{u(x, 0)}{s}  +\frac{1}{s^{\alpha}} \mathcal{ L} \left[ f\left(x, u(a_0 x, b_0 t), \frac{\partial }{\partial x}u(a_1 x, b_1 t), \ldots, \frac{\partial^m }{\partial x^m}u(a_m x, b_m t) \right)\right]
\end{equation}
Inverse Laplace transform of Eq.\eqref{eqn-RDT-SG-LT} leads
 \begin{equation}\label{eqn-RDT-SG-ILT}
u(x, t) = \psi(x) + \mathcal{L}^{-1}\left[\frac{1}{s^\alpha} \mathcal{L}\left[ f\left(x,u(a_0 x, b_0 t), \frac{\partial }{\partial x}u(a_1 x, b_1 t), \ldots, \frac{\partial^m }{\partial x^m}u(a_m x, b_m t) \right)\right]\right],
\end{equation}
where $\psi(x)$ denotes source term, usually the recommended initial conditions.

Let us assume from homotopy perturbation method that the basic solution of Eq. \eqref{eqn-RDT-SG-LT} in a power series:
\begin{equation}\label{eqn-basic}
 u^*(x,t)=\sum_{\imath =0}^{\infty}p^\imath u_\imath(x,t).
\end{equation}
From Eq. \eqref{eqn-basic} and \eqref{eqn-RDT-SG-ILT}, we get
 \begin{equation}\label{eqn-RDT-SG-ILT-hpm1}
 \begin{split}
&\sum_{r=0}^{\infty}p^r u_r(x,t)=u(x,0)+\\
&p \left[\mathcal{L}^{-1}\left\{\frac{1}{s^\alpha} \mathcal{L}\left\{ f\left(x, t,\sum_{r=0}^{\infty} u_r(a_0 x, b_0 t), \frac{\partial }{\partial x}\sum_{r=0}^{\infty} u_r(a_1 x, b_1 t), \ldots \frac{\partial^m }{\partial x^m} \sum_{r=0}^{\infty}u_r(a_m x, b_m t) \right)\right\} \right\} \right]
\end{split}
\end{equation}
On equating like powers of $p$, we get
\begin{equation} \label{eqn-RDT-SG-ILT-hpm11}
\left.\begin{split}
& p^0 : u_0(x,t)=\psi(x)\\
&p^1 :u_1(x,t)=\mathcal{L}^{-1}\left[\frac{1}{s^\alpha} \mathcal{L}\left[ f\left(x, t, u_0(a_0 x, b_0 t), \frac{\partial }{\partial x}u_0(a_1 x, b_1 t), \ldots, \frac{\partial^m }{\partial x^m}u_0(a_m x, b_m t) \right)\right]\right]\\
& \qquad \qquad \qquad  \vdots  \qquad \qquad \qquad  \qquad \vdots
 \end{split} \right\}
\end{equation}
For $p=1$, an approximate solution is given by
\begin{equation}\label{eqn-approx}
 u(x,t)=\sum_{\imath =0}^{\infty} u_\imath(x,t).
\end{equation}
\subsection{Convergence analysis and error estimate}\label{sec-conver.}
This sections studies the convergence of the HPTM solution and the error estimate.
\begin{theorem}\label{conv}
Let $0< \gamma <1$ and let $u_n(x, t), u(x, t)$ are in Banach space $(\mathcal{C}[0, 1], ||\cdot||)$. Then the series solution
$\sum_{n=0}^{\infty} u_n(x, t)$ from the sequence $\{u_n(x)\}_{n=0}^\infty$ converges to the solution of Eq. \eqref{eqn-PD} whenever $u_n(x)\leq \gamma u_{n-1)}(x)$ for all $n\in N$.

Moreover, the maximum absolute truncation error of the series solution \eqref{eqn-approx} for  Eq. \eqref{eqn-PD} is computed as
\begin{equation}\label{eqn-max-error}
 \left|\left|u(x,t)-\sum_{\imath =0}^{\ell} u_\imath(x,t)\right|\right|\leq \frac{\gamma^{\ell+1}}{1-\gamma}||u_0(x,t)||<\frac{||u_0(x,t)||}{1-\gamma}.
\end{equation}
\end{theorem}
\begin{proof} The proof is similar to \cite[Theorem 4.1, 4.2]{SUE16}.
\end{proof}
\section{Application of HPTM for TFPDEs with proportional delay}\label{sec-numeric-stdy}
In this section, the effectiveness and validity of HPTM is illustrated by three test problems of  initial valued autonomous system of TFPDEs with proportional delay.
\begin{example}\label{ex1}
Consider initial values system of time-fractional order, generalized Burgers equation with proportional delay as given in \cite{SUE16}
\begin{equation} \label{eqn-ex1}
\left\{\begin{split}
& \mathcal{D}_t^{\alpha} u(x, t) = \frac{\partial^2}{\partial x^2} u(x, t)+u\(\frac{x}{2}, \frac{t}{2}\) \frac{\partial }{\partial x}  u\left(x,\frac{t}{2}\right)+ \frac{1}{2}u(x, t) \\
& u(x, 0) = x,
\end{split} \right.
\end{equation}
Taking Laplace transformation of Eq. \eqref{eqn-ex1}, we get
\begin{equation}\label{eqn-ex1-LT}
\mathcal{U}(x, s)=\frac{x}{s} +\frac{1}{s^\alpha}\mathcal{L}\left[ \frac{\partial^2}{\partial x^2} u(x, t)+u\(\frac{x}{2}, \frac{t}{2}\) \frac{\partial }{\partial x}  u\left(x,\frac{t}{2}\right)+ \frac{1}{2}u(x, t)\right]
\end{equation}
Now, inverse Laplace transform with basic solution \eqref{eqn-basic} leads to
\begin{equation}\label{eqn-ex1-ILT-HPM}
\begin{split}
 &\sum \limits_{r = 0}^\infty  p^r u_r(x,t)=x+ \\
 &p \left[\mathcal{L}^{-1}\left\{\frac{1}{s^\alpha}\mathcal{L}\left( \frac{\partial^2}{\partial x^2} \(\sum\limits_{r = 0}^\infty u_r(x, t)\)+\sum\limits_{r = 0}^{\infty} u_r\(\frac{x}{2}, \frac{t}{2}\) \frac{\partial }{\partial x}  u_{k-r}\left(x,\frac{t}{2}\right)+ \frac{1}{2}\sum\limits_{r = 0}^\infty u_r(x, t)\right)\right\}\right]
 \end{split}
\end{equation}
On comparing the coefficient of like powers of $p$, we get
\begin{equation} \label{eqn-ex1-HPTMa}
\left.\begin{split}
p^0 :  u_0(x,t)&=x,\\
p^1:u_1(x,t)&= \mathcal{L}^{-1}\left[\frac{1}{s^\alpha}\mathcal{L}\left[ \frac{\partial^2}{\partial x^2} \( u_0(x, t)\)+ u_0\(\frac{x}{2}, \frac{t}{2}\) \frac{\partial }{\partial x}  u_0\left(x,\frac{t}{2}\right)+ \frac{1}{2}u_0(x, t)\right]\right]\\
&= \frac{x t^\alpha}{\Gamma{(1+\alpha)}} \\
p^2: u_2(x,t)&= \mathcal{L}^{-1}\left[\frac{1}{s^\alpha}\mathcal{L}\left\{ \frac{\partial^2}{\partial x^2} \( u_1(x, t)\)+\frac{1}{2}u_1(x, t)\right.\right.\\
&\left.\left.+ u_0\(\frac{x}{2}, \frac{t}{2}\) \frac{\partial }{\partial x}  u_1\left(x,\frac{t}{2}\right)+ u_1\(\frac{x}{2}, \frac{t}{2}\) \frac{\partial }{\partial x}  u_0\left(x,\frac{t}{2}\right)\right\}\right]\\
&= \frac{(1+2^{1-\alpha})x t^{2\alpha}}{2\Gamma{(1+2\alpha)}}\\
 \end{split} \right.
\end{equation}
\begin{equation} \label{eqn-ex1-HPTMb}
\left.\begin{split}
p^3:u_3(x,t)&= \mathcal{L}^{-1}\left[\frac{1}{s^\alpha}\mathcal{L}\left[ \frac{\partial^2}{\partial x^2} \( u_2(x, t)\)+ \frac{u_2(x, t)}{2} +u_2\(\frac{x}{2}, \frac{t}{2}\) \frac{\partial }{\partial x}  u_0\left(x,\frac{t}{2}\right)\right.\right.\\
&\left.\left.+ u_0\(\frac{x}{2}, \frac{t}{2}\) \frac{\partial }{\partial x}  u_2\left(x,\frac{t}{2}\right)+ u_1\(\frac{x}{2}, \frac{t}{2}\) \frac{\partial }{\partial x}  u_1\left(x,\frac{t}{2}\right)\right]\right]\\
&=\frac{x t^{3\alpha}}{4 \Gamma{(1+3\alpha)}} \left\{1+ 2^{1-\alpha} + 2^{1-2\alpha} + 2^{2-3\alpha}  + \frac{ \Gamma{(1+2\alpha)}}{\Gamma{(1+\alpha)}^2} ~~ 2^{1-2\alpha}\right\}\\
p^4:u_4(x,t)&= \mathcal{L}^{-1}\left[\frac{1}{s^\alpha}L\left[ \frac{\partial^2}{\partial x^2} u_3(x, t)+ \frac{u_3(x, t)}{2}+ u_0\(\frac{x}{2}, \frac{t}{2}\) \frac{\partial }{\partial x} u_3\left(x,\frac{t}{2}\right)+ u_1\(\frac{x}{2}, \frac{t}{2}\)\right.\right.\\
&\left.\left. \times\frac{\partial }{\partial x}  u_2\left(x,\frac{t}{2}\right) + u_2\(\frac{x}{2}, \frac{t}{2}\) \frac{\partial }{\partial x}  u_1\left(x,\frac{t}{2}\right)+u_3\(\frac{x}{2}, \frac{t}{2}\) \frac{\partial }{\partial x}  u_0\left(x,\frac{t}{2}\right)\right]\right]\\
&= \frac{{x t^{4\alpha}}}{8\Gamma{(1+4\alpha)}}\left\{1+2^{9-6\alpha}+ 2^{8-5\alpha}+3\times 2^{7-3\alpha}+2^{7-2\alpha}
+2^{7-\alpha}+2^{8-4\alpha}\right. \\
& \left.+(2^{8-5\alpha}+2^{7-2\alpha}) \frac{\Gamma{(1+2\alpha)}}{\Gamma{(1+\alpha)^2}}+(2^{9-4\alpha}+2^{8-3\alpha})\frac{\Gamma{(1+3\alpha)}}{\Gamma{(1+\alpha)}\Gamma{(1+2\alpha)}} \right\}\\
& \qquad \qquad \qquad  \vdots  \qquad \qquad \qquad  \qquad \vdots
 \end{split} \right.
\end{equation}
Therefore the solution for Eq. \eqref{eqn-ex1} is
\begin{equation}\label{eqn-ex1-HP-ILT-SOLN}
 u(x,t)=u_0(x,t)+u_1(x,t)+u_2(x,t)+u_3(x,t)+u_4(x,t)+\ldots
\end{equation}
The same solution is is obtained by Sarkar et al. \cite{SUE16}. In particular, for $\alpha=1$, the seventh order solution is
\begin{equation}\label{eqn-ex1-HP-ILT-exact-SOLN}
 u(x,t)=x\left( 1+t+\frac{t^2}{2}+\frac{t^3}{6}+\frac{t^4}{24}+\frac{t^5}{120}+\frac{t^6}{720}+\frac{t^7}{5040}\right)
\end{equation}
which is same as obtained by DTM and RDTM \cite {AG11}, and is a closed form of the exact solution $u(x, t) = x\exp(t)$. The HPTM solution for $\alpha=1$ is reported in Table \ref{tab1.1}. The surface solution behavior of $u(x,t)$ for different values of $\alpha =0.8, 0.9, 1.0$ is depicted in Fig. \ref{fig1.1}, and the plots of the solution for $x=1$ at different time intervals $t\leq 1$ is depicted in Fig \ref{fig1.11}. It is found that the results are agreed well with HPM as well as DTM solutions and approaching to the exact solution.

\begin{table}[t!]
\caption{Approximate HPTM solution of Example \ref{ex1} with first six terms for at $\alpha =1.0$}  \label{tab1.1}
\centering
\begin{tabular}{llllllllll}
\toprule
$x$	        &{}& $t$  &{}    & Exact         &{}&\multicolumn{2}{l}{HPTM}&{}&  \\ 
{}	        &{}&{}           & {} &          &{}&solution         &{}& $E_{abs}$ \\
\midrule
   $ 0.25$	\qquad \qquad &{}&$0.25$&{}    &0.321006	   &{}&0.321004     &{}&2.122401E-06\\
	        & \qquad\qquad&$0.50$\qquad  \qquad\qquad &{} 	 &0.412180  \qquad \qquad \qquad    &{}&0.412109  \qquad \qquad \qquad    &{}&7.094268E-05\\
	        &{}&$0.75$&{}    &0.529250	   &\qquad&0.528686     &\qquad&5.634807E-04\\
	        &{}&$1.00$&{}    &0.679570     &{}&0.677083 	&{}&2.487124E-03\\
    $ 0.5$	&{}&$0.25$&{}    &0.642012     &{}&0.642008 	&{}&4.244802E-06\\
	        &{}&$0.50$&{}    &0.824361     &{}&0.824219	    &{}&1.418854E-04\\
	        &{}&$0.75$&{}    &1.058500     &{}&1.057373   	&{}&1.126961E-03\\
	        &{}&$1.00$&{}	 &1.359141     &{}&1.354167	    &{}&4.974248E-03\\
    $0.75$	&{}&$0.25$&{}    &0.963019     &{}&0.963012  	&{}&6.369688E-06\\
	        &{}&$0.50$&{}    &1.236541	   &{}&1.236328	    &{}&2.128250E-04\\
	        &{}&$0.75$&{}    &1.587750     &{}&1.586060	    &{}&   1.690020E-03     \\
	        &{}&$1.00$&{}    &2.038711     &{}&2.031250  	&{}&7.461370E-03\\
\bottomrule
\end{tabular}
\end{table}
\begin{figure}[!t]
\includegraphics[width=5.30cm,height=5.20 cm]{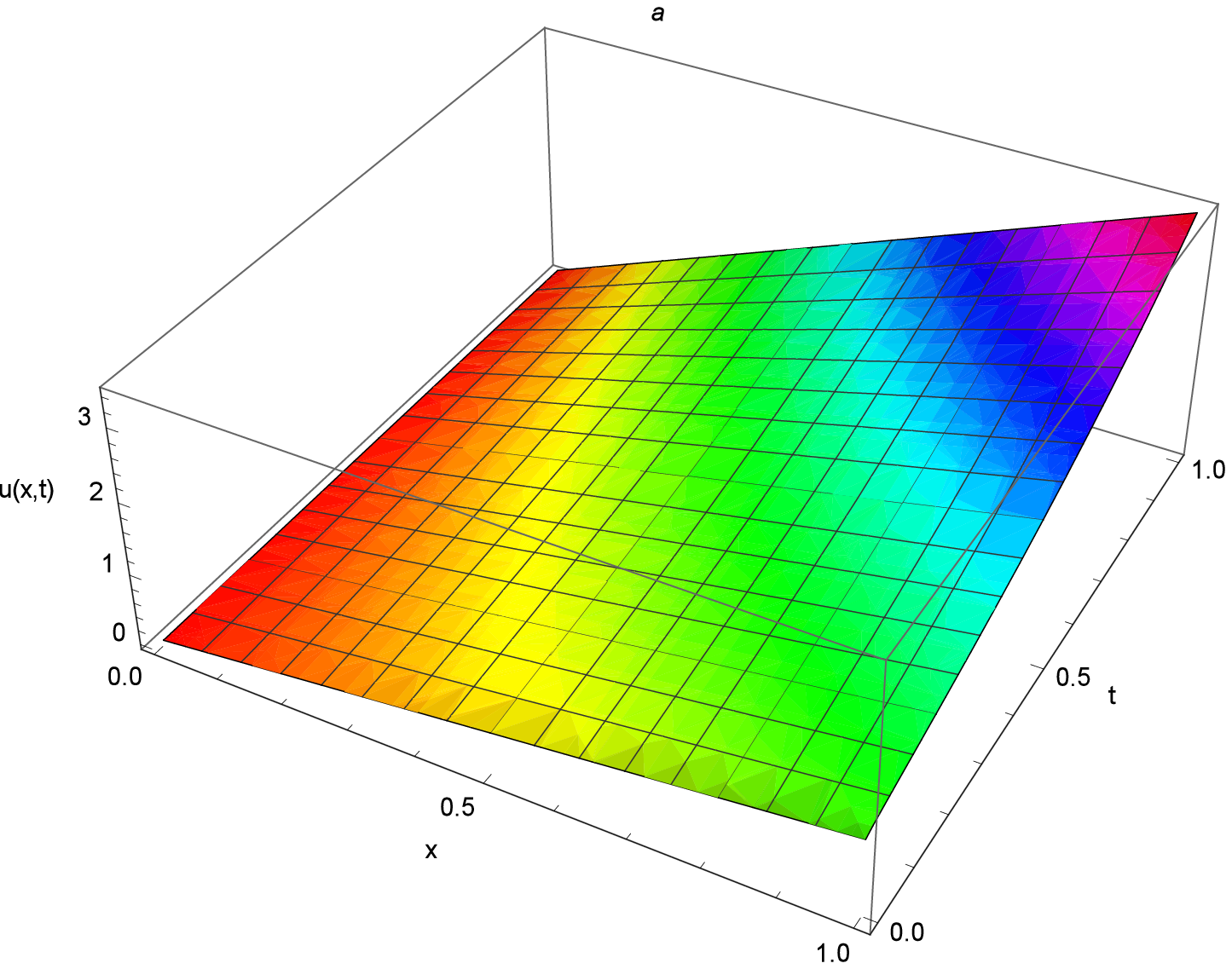}
\includegraphics[width=5.30cm,height=5.20 cm]{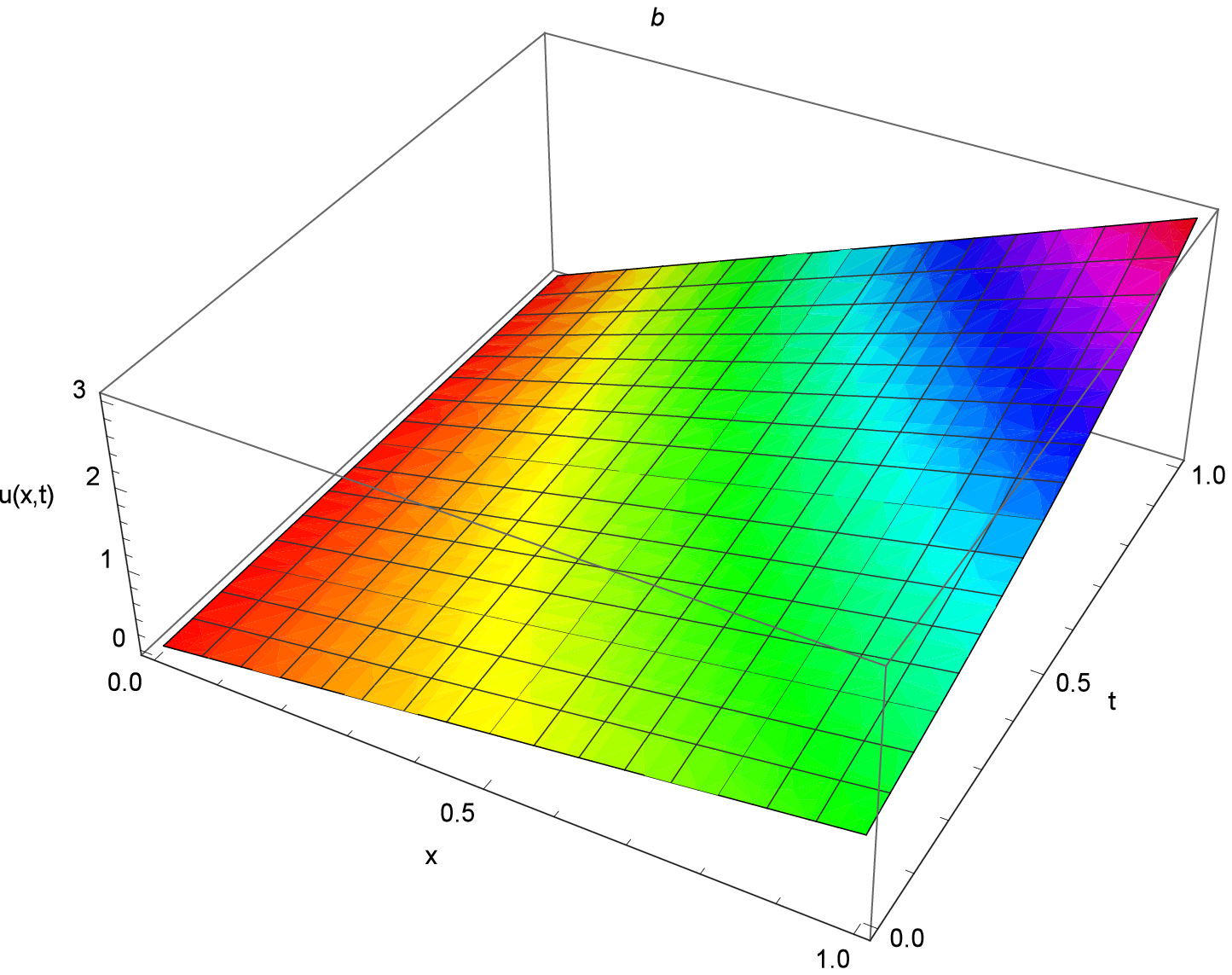}
\includegraphics[width=5.530cm,height=5.20 cm]{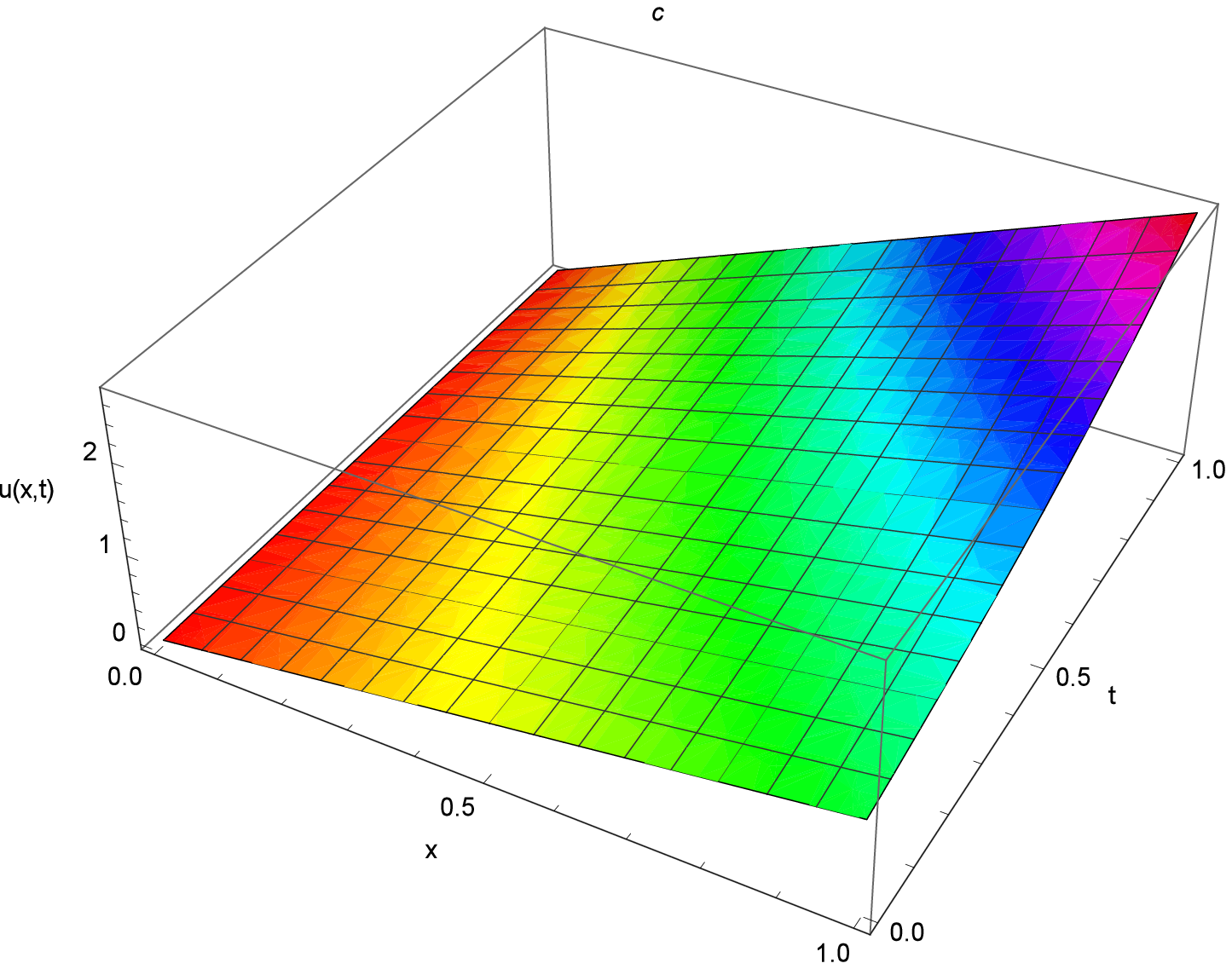}
\caption{The surface HPTM solution behavior of $u$ of Example \ref{ex1} for (a) $\alpha =0.8$; (b
) $\alpha =0.9$; (c) $\alpha =1.0$} \label{fig1.1}
\end{figure}
\begin{figure}[!t]
\centering
\includegraphics[width=9.30cm,height=6.20 cm]{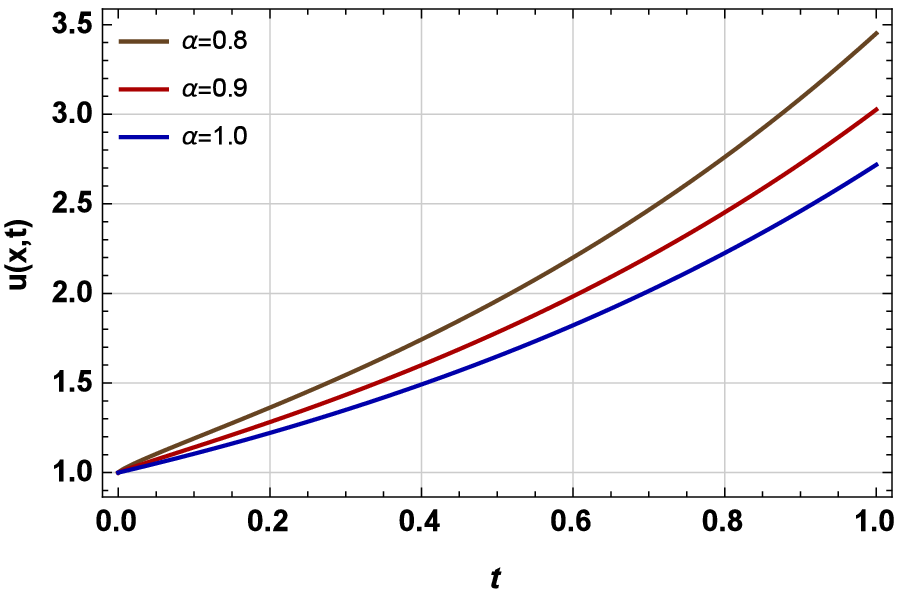}
\caption{Plots of HPTM solution $u(x,t)$ of Example \ref{ex1} for$\alpha=0.8, 0.9, 1.0$, $t\in (0,1); x=1$} \label{fig1.11}
\end{figure}
\end{example}

\begin{example}\label{ex2}
Consider initial value TFPDE with proportional delay as given in \cite{AG11,SUE16}
\begin{equation} \label{eqn-ex2}
\left\{\begin{split}
& \mathcal{D}_t^{\alpha} u(x, t) = u\(x, \frac{t}{2}\) \frac{\partial^2 }{\partial x^2}  u\left(x,\frac{t}{2}\right) - u(x, t) \\
& u(x, 0) = x^2,
\end{split} \right.
\end{equation}
Taking Laplace transform of Eq. \eqref{eqn-ex2}, we get
\begin{equation}\label{eqn-ex2-LT}
 \mathcal{U}(x, s)=\frac{x^2}{s} +\frac{1}{s^\alpha}\mathcal{L}\left[u\(x, \frac{t}{2}\) \frac{\partial^2 }{\partial x^2}  u\left(x,\frac{t}{2}\right) - u(x, t)\right]
\end{equation}
The inverse Laplace transform of Eq. \eqref{eqn-ex2-LT} leads to
\begin{equation}\label{eqn-ex2-ILT}
  u(x,t)=x^2 +\mathcal{L}^{-1}\left[\frac{1}{s}x +\frac{1}{s^\alpha}\mathcal{L}\left\{u\(x, \frac{t}{2}\) \frac{\partial^2 }{\partial x^2}  u\left(x,\frac{t}{2}\right) - u(x, t)\right\}\right]
\end{equation}
Eq. \eqref{eqn-ex2-ILT} with basic solution \eqref{eqn-basic} leads to
\begin{equation}\label{eqn-ex2-ILT-HPM}
 \sum\limits_{n = 0}^\infty  p^n u_n(x,t)=x^2+p\left[\mathcal{L}^{-1}\left\{\frac{1}{s^\alpha}\mathcal{L}\left(  \sum\limits_{n = 0}^{\infty} u_n\(x, \frac{t}{2}\) \frac{\partial^2 }{\partial x^2}  u_{k-n}\left(x,\frac{t}{2}\right)  - \sum\limits_{n = 0}^\infty u_n(x, t)\right)\right\}\right]
\end{equation}
On comparing the coefficient of like powers of $p$, we get
\begin{equation} \label{eqn-ex2-HP-LTM}
\begin{split}
 p^0 : u_0(x,t)&=x^2,\\
p^1 :u_1(x,t)&= \mathcal{L}^{-1}\left[\frac{1}{s^\alpha}\mathcal{L}\left[  u_0\(x, \frac{t}{2}\) \frac{\partial^2 }{\partial x^2}  u_{0}\left(x,\frac{t}{2}\right)  -  u_0(x, t)\right]\right]\\
&= \frac{x^2 t^\alpha}{\Gamma{(1+\alpha)}} \\
p^2:u_2(x,t)&= \mathcal{L}^{-1}\left[\frac{1}{s^\alpha}\mathcal{L}\left[  u_0\(x, \frac{t}{2}\) \frac{\partial^2 }{\partial x^2}  u_{1}\left(x,\frac{t}{2}\right)  + u_1\(x, \frac{t}{2}\) \times\right.\right.\\
&\left.\left. \frac{\partial^2 }{\partial x^2}  u_{0}\left(x,\frac{t}{2}\right)-  u_1(x, t)\right]\right]= \frac{x^2 t^{2\alpha}(2^{2-\alpha}-1)}{\Gamma{(1+2\alpha)}}\\
p^3:u_3(x,t)&= \mathcal{L}^{-1}\left[\frac{1}{s^\alpha}\mathcal{L}\left[  u_0\(x, \frac{t}{2}\) \frac{\partial^2 }{\partial x^2}  u_{2}\left(x,\frac{t}{2}\right)  + u_2\(x, \frac{t}{2}\) \frac{\partial^2 }{\partial x^2}  u_{0}\left(x,\frac{t}{2}\right)\right]\right]\\
&+ \mathcal{L}^{-1}\left[\frac{1}{s^\alpha}\mathcal{L}\left[  u_1\(x, \frac{t}{2}\) \frac{\partial^2 }{\partial x^2}  u_{1}\left(x,\frac{t}{2}\right)-  u_2(x, t)\right]\right]\\
&=\frac{x^2 t^{3\alpha}}{ \Gamma{(1+3\alpha)}}
\left\{1- 2^{2-\alpha} - 2^{2-2\alpha} + 2^{4-3\alpha}  + \frac{ \Gamma{(1+2\alpha)}}{\Gamma{(1+\alpha)}^2} ~~ 2^{1+\alpha}\right\}\\
p^4:u_4(x,t)&= \mathcal{L}^{-1}\left[\frac{1}{s^\alpha}\mathcal{L}\left\{  u_0\(x, \frac{t}{2}\) \frac{\partial^2 }{\partial x^2}  u_{3}\left(x,\frac{t}{2}\right)  + u_2\(x, \frac{t}{2}\) \frac{\partial^2 }{\partial x^2}  u_{1}\left(x,\frac{t}{2}\right)\right.\right.\\
&+ \left.\left. u_1\(x, \frac{t}{2}\) \frac{\partial^2 }{\partial x^2}  u_{2}\left(x,\frac{t}{2}\right)+ u_3\(x, \frac{t}{2}\) \frac{\partial^2 }{\partial x^2}  u_{0}\left(x,\frac{t}{2}\right)-  u_3(x, t)\right\}\right]\\
&= \frac{{x^2 t^{4\alpha}}}{\Gamma{(1+4\alpha)}}\left(1-2^{2-\alpha}- 2^{2-2\alpha}+3\times2^{2-3\alpha}+2^{4-4\alpha}
+2^{4-5\alpha}-2^{6-6\alpha}\right) \\
&+\frac{{x^2 t^{4\alpha}}}{\Gamma{(1+4\alpha)}}\left( \frac{(2^{1+4\alpha}-2^{3+2\alpha})\Gamma{(1+2\alpha)}}{\Gamma{(1+\alpha)^2}}+\frac{(2^{2+3\alpha}-2^{4+2\alpha})\Gamma{(1+3\alpha)}}{\Gamma{(1+\alpha)}\Gamma{(1+2\alpha)}} \right)\\
& \qquad \qquad \qquad  \vdots  \qquad \qquad \qquad  \qquad \vdots
 \end{split}
\end{equation}
Thus, the solution for Eq. \eqref{eqn-ex2} is given by
\begin{equation}\label{eqn-ex2-HP-ILT-SOLN}
 u(x,t)=u_0(x,t)+u_1(x,t)+u_2(x,t)+u_3(x,t)+u_4(x, t)+\ldots
\end{equation}
which is the required exact solution, same solution is obtained by Sarkar et al \cite{SUE16}. In particular for $\alpha=1$, the seventh order solution is obtained as \begin{equation}\label{eqn-ex2-HP-ILT-exact-SOLN}
u(x,t)=x^2\left( 1+t+\frac{t^2}{2}+\frac{t^3}{6}+\frac{t^4}{24}+\frac{t^5}{120}+\frac{t^6}{720}+\frac{t^7}{5040}\right)
\end{equation}
which is same as obtained by DTM and RDTM \cite {AG11}, and is a closed form of the exact solution $u(x, t) = x^2 \exp(t)$. The HPTM solution for $\alpha=1$ is reported in Table \ref{tab2.1}.  The surface solution behavior of $u(x,t)$ for different values of $\alpha =0.8, 0.9, 1.0$ is depicted in Fig. \ref{fig2.1}, and the plots of the solution for $x=1$ at different time intervals $t\leq 1$ is depicted in Fig \ref{fig2.11}. It is found that the proposed HPTM results are agreed well with HPM as well as DTM solutions and approaching to the exact solution.

\begin{figure}[!t]
\includegraphics[width=5.30cm,height=5.20 cm]{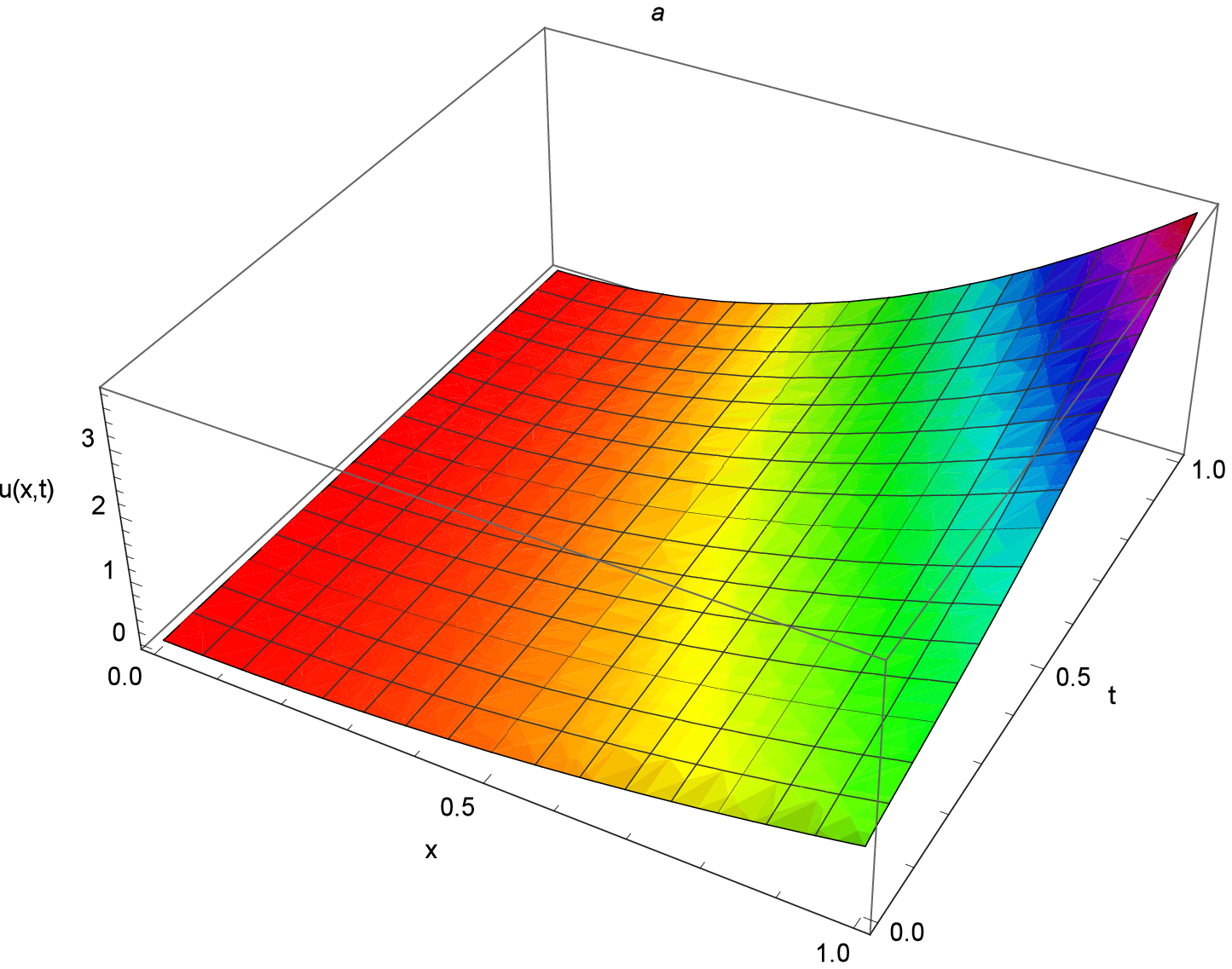}
\includegraphics[width=5.30cm,height=5.20 cm]{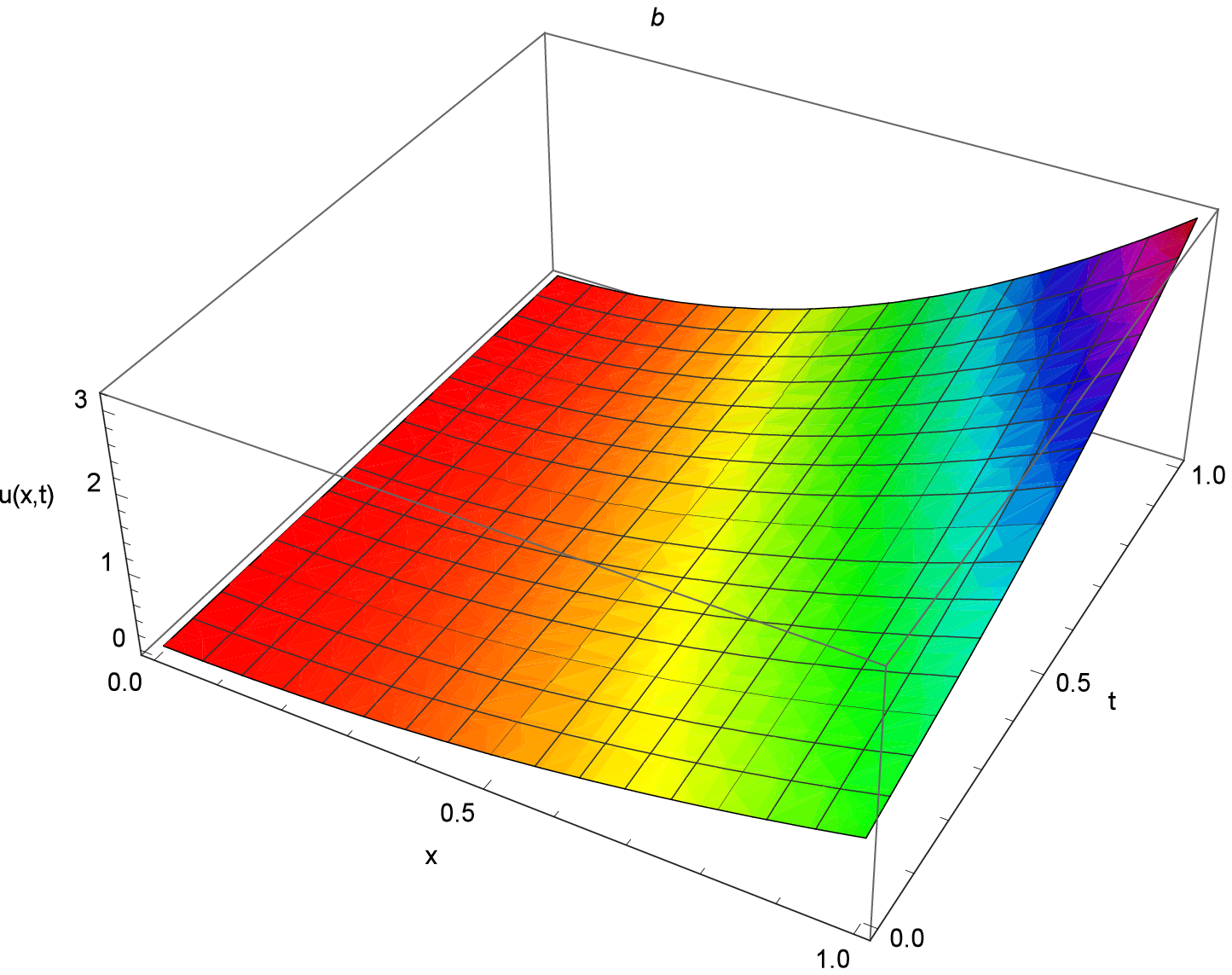}
\includegraphics[width=5.530cm,height=5.20 cm]{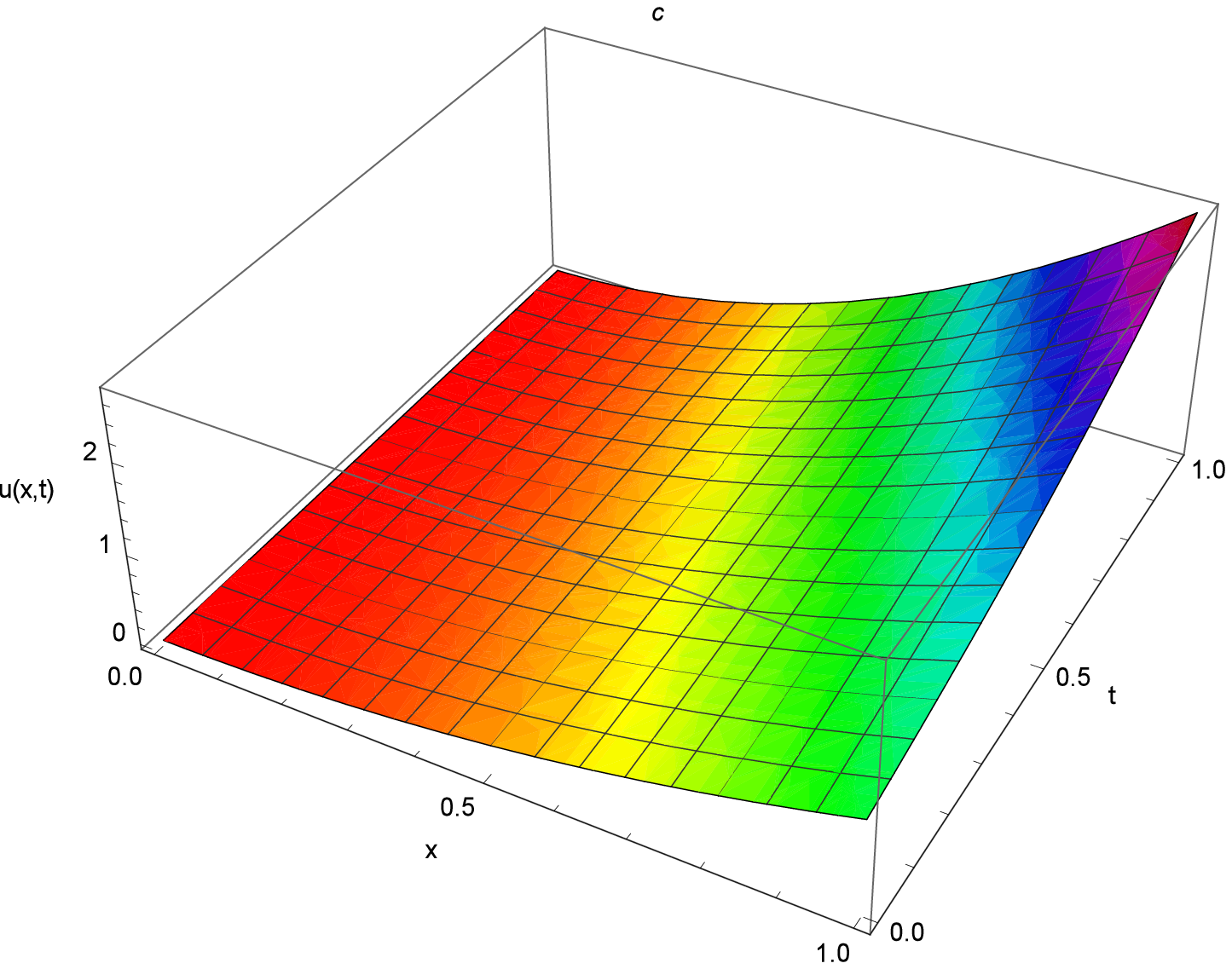}
\caption{The surface solution behavior of $u$ of Example \ref{ex2} for (a) $\alpha =0.8$; (b
) $\alpha =0.9$; (c) $\alpha =1.0$} \label{fig2.1}
\end{figure}

\begin{figure}[!t]
\centering
\includegraphics[width=9.30cm,height=6.20 cm]{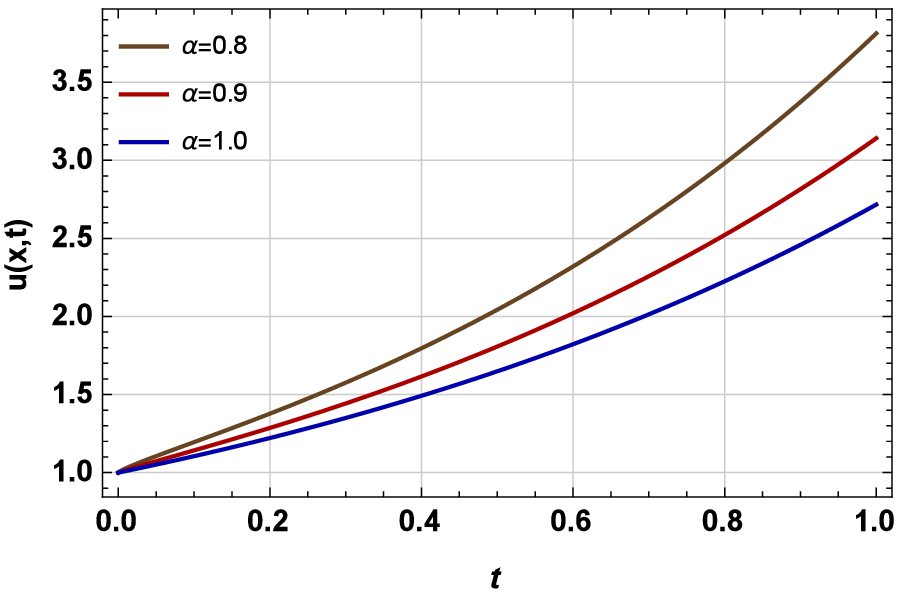}
\caption{Plots of  HPTM solution $u(x,t)$ of Example \ref{ex2} for$\alpha=0.8, 0.9, 1.0$  $t(0,1);x=1$} \label{fig2.11}
\end{figure}

\begin{table}[t!]
\caption{Approximate HPTM solution of Example \ref{ex2} with first six terms for at $\alpha =1.0$}  \label{tab2.1}
\centering
\begin{tabular}{llllllllll}
\toprule
$x$	   &{}& $t$   &{}&      Exact   &{}&\multicolumn{2}{l}{HPTM}&{}&  \\ 
{}	   &{}&   {}  &{}&              &{}&solution         &{}& $E_{abs}$ \\
\midrule										
$0.25$\qquad  \qquad \qquad\qquad& \qquad \qquad&$0.25$  \qquad\qquad&{}&	0.0802516\qquad\qquad \qquad&{}&	0.0802516 \qquad\qquad\qquad&{}&	7.812108E-10&	 \\
	   &{}&$0.50$ &{}&	0.1030451&{}&	0.1030450&{}&	1.032903E-07&	\\
	   &{}&$0.75$ &{}&  0.1323125&{}&	0.1323107&{}&	1.824464E-06&	\\
	   &{}&$1.00$ &{}&	0.1698926&{}&	0.1698785&{}&	1.414206E-05&	\\
$0.50$ &{}&$0.25$ &{}&	0.3210064&{}&	0.3210064&{}&	3.124843E-09&	\\
	   &{}&$0.50$ &{}&	0.4121803&{}&	0.4121799&{}&	4.131611E-07&   \\
	   &{}&$0.75$ &{}&	0.5292500&{}&	0.5292427&{}&	7.297854E-06&	\\
	   &{}&$1.00$ &{}&	0.6795705&{}&	0.6795139&{}&	5.656823E-05&	\\
$0.75$ &{}&$0.25$ &{}&	0.7222643&{}&	0.7222643&{}&	7.030897E-09&	\\
	   &{}&$0.50$ &{}&	0.9274057&{}&	0.9274048&{}&	9.296126E-07&	\\
	   &{}&$0.75$ &{}&	1.1908130&{}&	1.1907963&{}&	1.642017E-05&	\\
	   &{}&$1.00$ &{}&	1.5290340&{}&	1.5289062&{}&	1.272785E-04&	\\
\bottomrule
\end{tabular}
\end{table}
\end{example}

\begin{example}\label{ex3}
Consider initial value TFPDE with proportional delay as given in \cite{AG11,SUE16}
\begin{equation} \label{eqn-ex3}
\left\{\begin{split}
& \mathcal{D}_t^{\alpha} u(x, t) = \frac{\partial^2 }{\partial x^2} u\left(\frac{x}{2},\frac{t}{2}\right) ~\frac{\partial }{\partial x}  u\(\frac{x}{2}, \frac{t}{2}\)-\frac{1}{8} \frac{\partial }{\partial x}  u\(x, t\) - u(x, t) \\
& u(x, 0) = x^2,
\end{split} \right.
\end{equation}
\begin{table}[t!]
\caption{Approximate HPTM solution of Example \ref{ex3} with first six terms for at $\alpha =1.0$}  \label{tab3.1}
\centering
\begin{tabular}{llllllllll}
\toprule
$x$	   &{}& $t$   &{}& Exact  &{}&\multicolumn{2}{l}{HPTM}&{}&   \\
{}	   &{}& {}  &{}& {}  &{}&solution&{}& $E_{abs}$ \\
\midrule										
$0.25$ \qquad \qquad&\qquad \qquad&$0.25$ \qquad \qquad&{}&	4.867505E-02 \qquad \qquad&\qquad \qquad&	4.867505E-02\qquad\qquad \qquad &{}&	 7.338727E-10& 	 \\
	   &{}&$0.50$ &{}&	3.790817E-02&{}&	3.790826E-02&{}&	9.114643E-08&	\\
	   &{}&$0.75$ &{}&   2.952291E-02&{}&	2.952442E-02&{}&	1.512146E-06&   \\
	   &{}&$1.00$ &{}&	2.299247E-02&{}&	2.300347E-02&{}&    1.100715E-05&	\\
$0.50$ &{}&$0.25$ &{}&	1.947002E-01&{}&	1.947002E-01&{}&	2.935491E-09&	\\
	   &{}&$0.50$ &{}&	1.516327E-01&{}&	1.516330E-01&{}&	3.645857E-07&	\\
	   &{}&$0.75$ &{}&	1.180916E-01&{}&	1.180977E-01&{}&	6.048582E-06&	\\
	   &{}&$1.00$ &{}&	9.196986E-02&{}&	9.201389E-02&{}&	4.402860E-05&	\\
$0.75$ &{}&$0.25$ &{}&  4.380754E-01&{}&	4.380754E-01&{}&	6.604854E-09&	\\
	   &{}&$0.50$ &{}&	3.411735E-01&{}&	3.411743E-01&{}&	8.203179E-07&	\\
	   &{}&$0.75$ &{}&	2.657062E-01&{}&	2.657198E-01&{}&	1.360931E-05&	\\
	   &{}&$1.00$ &{}&	2.069322E-01&{}&	2.070313E-01&{}&	9.906434E-05&	\\
\bottomrule
\end{tabular}
\end{table}
Taking Laplace transform of Eq. \eqref{eqn-ex3}, we get
\begin{equation}\label{eqn-ex3-LT}
 \mathcal{U}(x, s)=\frac{x^2}{s} +\frac{1}{s^\alpha}\mathcal{L}\left[\frac{\partial^2 }{\partial x^2} u\left(\frac{x}{2},\frac{t}{2}\right) ~\frac{\partial }{\partial x}  u\(\frac{x}{2}, \frac{t}{2}\)-\frac{1}{8} \frac{\partial }{\partial x}  u\(x, t\) - u(x, t)\right]
\end{equation}
The inverse Laplace transform leads to
\begin{equation}\label{eqn-ex3-ILT}
  u(x,t)=x^2 +\mathcal{L}^{-1}\left[\frac{1}{s}x +\frac{1}{s^\alpha}\mathcal{L}\left\{\frac{\partial^2 }{\partial x^2} u\left(\frac{x}{2},\frac{t}{2}\right) ~\frac{\partial }{\partial x}  u\(\frac{x}{2}, \frac{t}{2}\)-\frac{1}{8} \frac{\partial }{\partial x}  u\(x, t\) - u(x, t)\right\}\right]
\end{equation}
Homotopy perturbation transform method on Eq.\eqref{eqn-ex3-ILT} leads to
\begin{equation}\label{eqn-ex3-ILT-HPM}
\begin{split}
&\sum\limits_{n = 0}^\infty  p^n u_n(x,t)=x^2+\\
& p\left(\mathcal{L}^{-1}\left[\frac{1}{s^\alpha}\mathcal{L}\left[  \sum\limits_{n = 0}^{k=\infty}\frac{\partial^2 u_n }{\partial x^2}\left(\frac{x}{2},\frac{t}{2}\right) ~\frac{\partial u_{k-n}}{\partial x}  \(\frac{x}{2}, \frac{t}{2}\) -\frac{1}{8} \frac{\partial }{\partial x}\sum\limits_{n = 0}^\infty  u_n\(x, t\) - \sum\limits_{n = 0}^\infty u_n(x, t)\right]\right]\right)
\end{split}
\end{equation}
On comparing the coefficient of like powers of $p$, we get
\begin{equation} \label{eqn-ex3-HP-LTM}
\begin{split}
 p^0 :& u_0(x,t)=x^2,\\
p^1 :&u_1(x,t)= \mathcal{L}^{-1}\left[\frac{1}{s^\alpha}\mathcal{L}\left[ \frac{\partial^2 }{\partial x^2} u_0\left(\frac{x}{2},\frac{t}{2}\right) ~\frac{\partial }{\partial x}  u_{0}\(\frac{x}{2}, \frac{t}{2}\)  -\frac{1}{8} \frac{\partial }{\partial x} u_0\(x, t\) -u_0(x, t)\right]\right]
\\&= \frac{-x^2 t^\alpha}{\Gamma{(1+\alpha)}} \\
p^2:&u_2(x,t)= \mathcal{L}^{-1}\left[\frac{1}{s^\alpha}\mathcal{L}\left[ \frac{\partial^2 }{\partial x^2} u_0\left(\frac{x}{2},\frac{t}{2}\right) ~\frac{\partial }{\partial x}  u_{1}\(\frac{x}{2}, \frac{t}{2}\)+u_1\left(\frac{x}{2},\frac{t}{2}\right) ~\frac{\partial }{\partial x}  u_{0}\(\frac{x}{2}, \frac{t}{2}\) \right.\right.\\
&\left.\left. -\frac{1}{8} \frac{\partial }{\partial x} u_1\(x, t\) -u_1(x, t)\right]\right]\\
&= t^{2\alpha}x\frac{(2^{1-\alpha}+2^{2}x+1)}{2\Gamma{(1+2\alpha)}}\\
p^3:&u_3(x,t)= L^{-1}\left[\frac{1}{s^\alpha}L\left[ \frac{\partial^2 }{\partial x^2} u_0\left(\frac{x}{2},\frac{t}{2}\right) ~\frac{\partial }{\partial x}  u_{2}\(\frac{x}{2}, \frac{t}{2}\)+\frac{\partial^2 }{\partial x^2} u_1\left(\frac{x}{2},\frac{t}{2}\right) ~\frac{\partial }{\partial x}  u_{1}\(\frac{x}{2}, \frac{t}{2}\)\right.\right.\\
&+ \left.\left. \frac{\partial^2 }{\partial x^2} u_2\left(\frac{x}{2},\frac{t}{2}\right) ~\frac{\partial }{\partial x}  u_{0}\(\frac{x}{2}, \frac{t}{2}\) -\frac{1}{8} \frac{\partial }{\partial x} u_2\(x, t\) -u_2(x, t)\right]\right]\\
&=\frac{ t^{3\alpha}}{ 2\Gamma{(1+3\alpha)}}
\left\{-1-2x^2-2^4+ 2^{-\alpha} + 2^{-2\alpha} + 2^{-3-\alpha}+2^{-3-2\alpha}+2^{-2-3\alpha}\right. \\
&\left.+2^{-1-2\alpha}x\frac{\Gamma(1+2\alpha)}{\Gamma(1+\alpha)^2} \right\}\\
p^4:&u_4(x,t)=\mathcal{L}^{-1}\left[\frac{1}{s^\alpha}\mathcal{L}\left[ \frac{\partial^2 u_0}{\partial x^2}\left(\frac{x}{2},\frac{t}{2}\right) \frac{\partial u_{3}}{\partial x}\(\frac{x}{2}, \frac{t}{2}\)+\frac{\partial^2 u_1}{\partial x^2}\left(\frac{x}{2},\frac{t}{2}\right)\frac{\partial u_{2}}{\partial x}\(\frac{x}{2}, \frac{t}{2}\)\right.\right.\\
&+\left.\left. \frac{\partial^2 u_2}{\partial x^2} \left(\frac{x}{2},\frac{t}{2}\right) \frac{\partial u_{1}}{\partial x}  \(\frac{x}{2}, \frac{t}{2}\)+\frac{\partial^2 u_3}{\partial x^2}\left(\frac{x}{2},\frac{t}{2}\right)\frac{\partial u_{0}}{\partial x} \(\frac{x}{2}, \frac{t}{2}\)\right.\right.\\
&\left.\left. -\frac{1}{8} \frac{\partial }{\partial x} u_3\(x, t\) -u_3(x, t)\right]\right]\\
&=\frac{ t^{4\alpha}}{ \Gamma{(1+4\alpha)}}
\left\{(3\times2^{-5}-2^{-3-\alpha}-2^{-3-2\alpha}+2^{-3-4\alpha}+2^{-3-5\alpha})\right\}\\
&+\frac{ t^{4\alpha}}{ \Gamma{(1+4\alpha)}}\left\{((-2^{-2-3\alpha}-2^{-2-2\alpha}-
2^{-2-\alpha}+3\times2^-3)x+x^2)\right\}\\
& +\frac{ t^{4\alpha}}{ \Gamma{(1+4\alpha)}}
-(-2^{-4-\alpha}+2^{-5-2\alpha}+2^{-2-2\alpha}x)\frac{\Gamma(1+2\alpha)}{\Gamma(1+\alpha)^2}\\
&-\frac{ t^{4\alpha}}{ \Gamma{(1+4\alpha)}}\left\{
2^{-4-4\alpha}(-2+2^\alpha+2^{3+\alpha}x)\frac{\Gamma(1+3\alpha)}
{\Gamma(1+2\alpha)\times \Gamma(1+4\alpha))} \right\}\\
& \qquad \qquad \qquad  \vdots  \qquad \qquad \qquad  \qquad \vdots
 \end{split}
\end{equation}
The required solution of Eq. \eqref{eqn-ex3} is
\begin{equation}\label{eqn-ex3-HP-ILT-SOLN}
 u(x,t)=u_0(x,t)+u_1(x,t)+u_2(x,t)+u_3(x,t)+\ldots
\end{equation}
which is a closed form to the exact solution and the solution obtained by Sarkar et al \cite{SUE16}.  In particular, for $\alpha=1$, the seventh order solution is obtained as
\begin{equation}
\begin{split}
u(x, t)&= x^2 \left(1-t + \frac{t^2}{3}-\frac{t^3}{6}+\frac{t^4}{24}-\frac{t^5}{120}+\frac{t^6}{720}-\frac{t^7}{5040}+\ldots \right)
\end{split}
\end{equation}
which is same as obtained by DTM and RDTM \cite {AG11} , and is a closed form of the exact solution $u(x, t) = x^2 \exp(-t)$. The HPTM solution for $\alpha =1.0$ is reported in Table \ref{tab3.1}. The solution behavior of $u$ for different values of $\alpha =0.8, 0.9, 1.0$ is depicted in Fig. \ref{fig3.1}, while the plots for $x=1$ at different time levels $t\leq 1$ is depicted in Fig. \ref{fig3.11}. This evident that HPTM solutions are agreed well with HPM as well as DTM solutions, and approaching to the exact solution.

\begin{figure}[!t]
\includegraphics[width=5.30cm,height=5.20 cm]{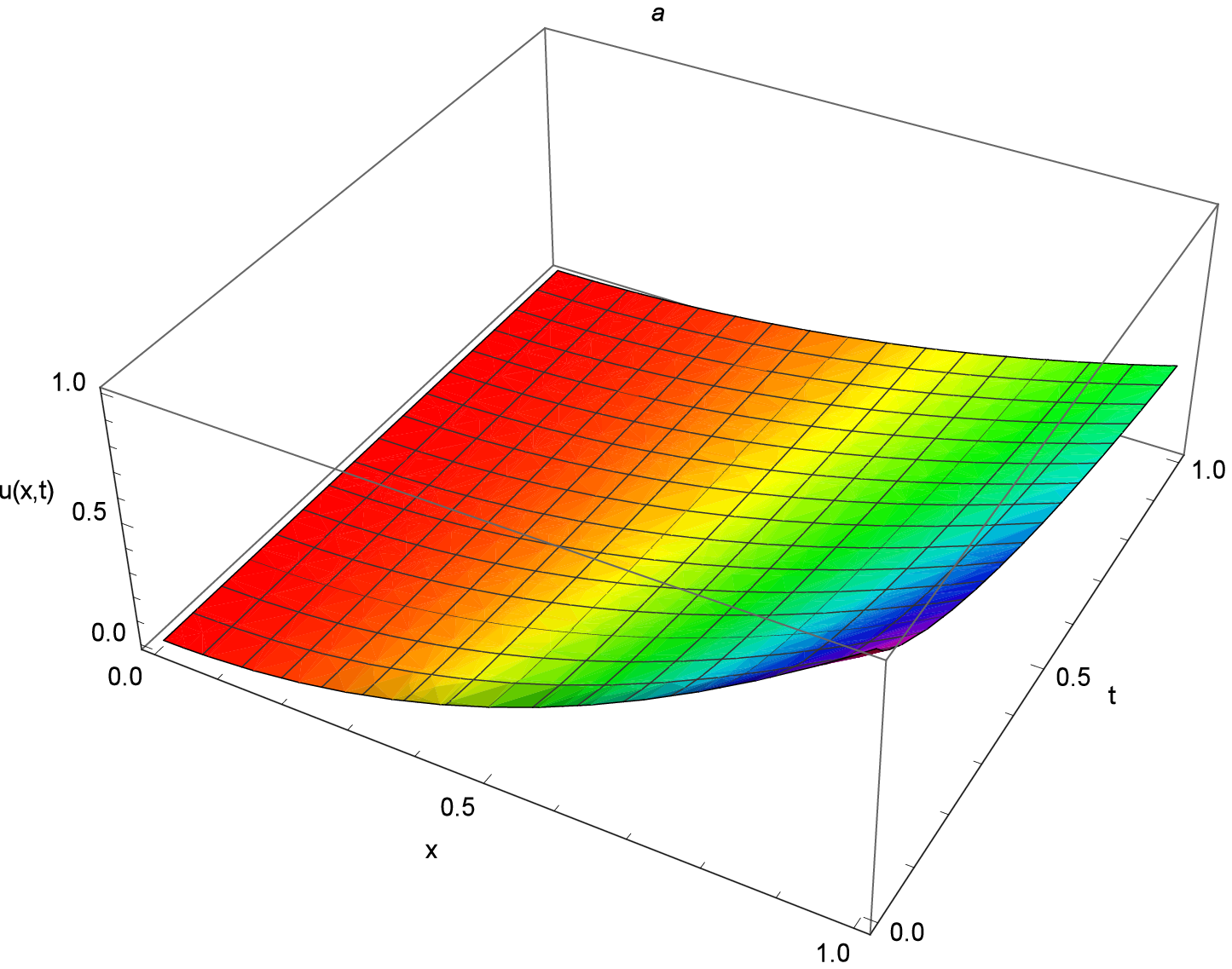}
\includegraphics[width=5.30cm,height=5.20 cm]{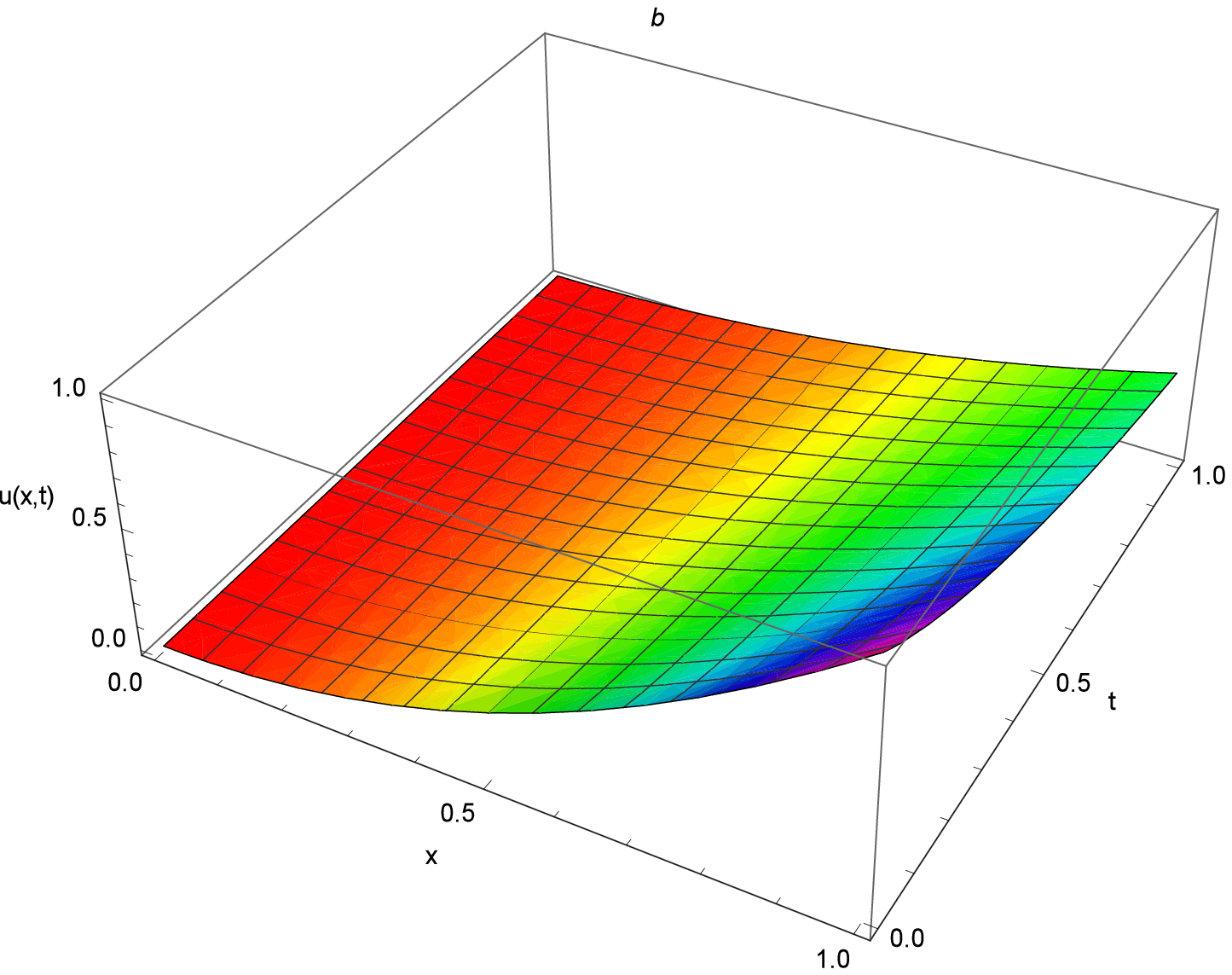}
\includegraphics[width=5.530cm,height=5.20 cm]{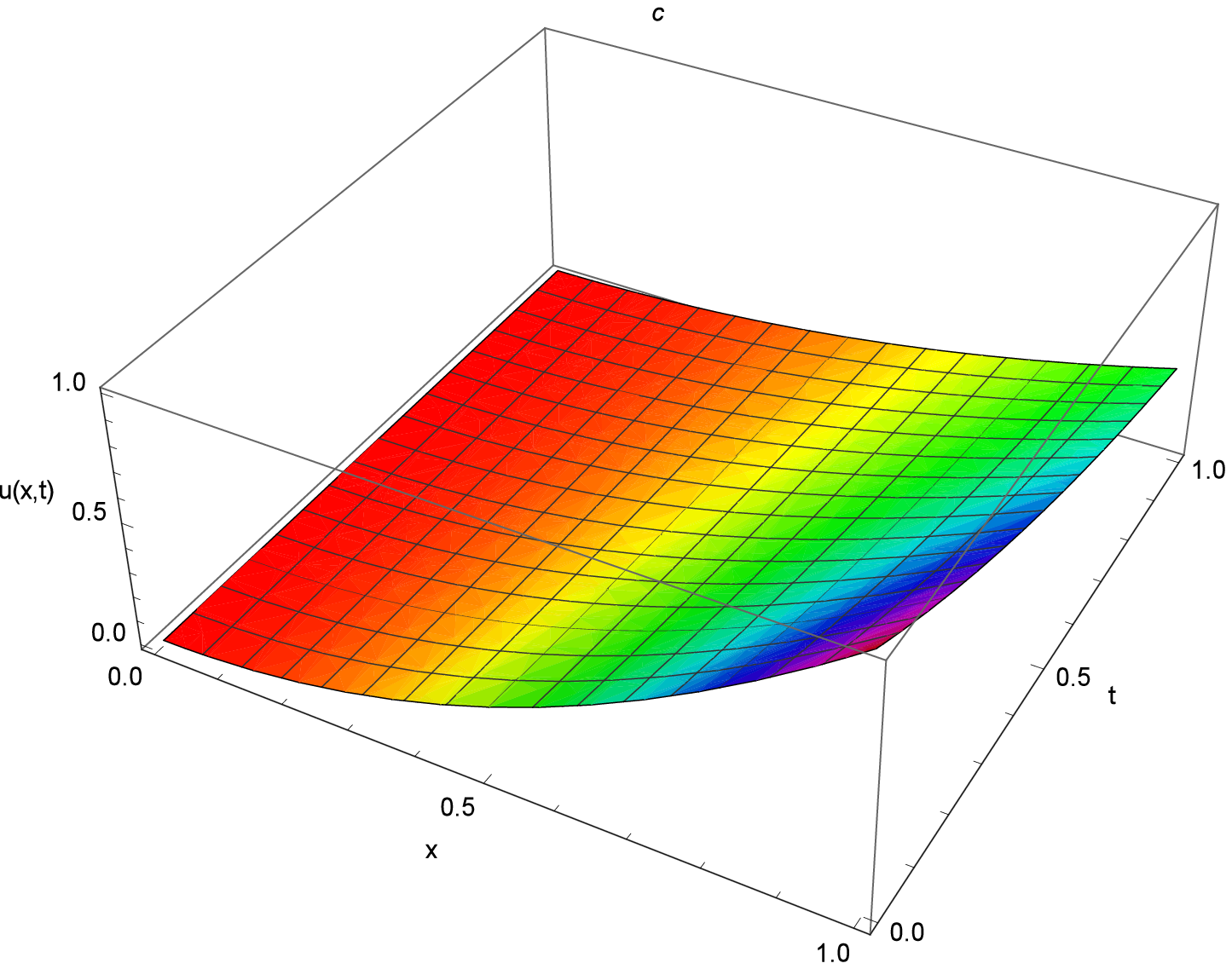}
\caption{The solution behavior of HPTM solution $u$ of Example \ref{ex3} for (a) $\alpha =0.8, 0.9, 1.0$; (b
) $\alpha =0.9$; (c) $\alpha =1.0$} \label{fig3.1}
\end{figure}

\begin{figure}[!t]
\centering
\includegraphics[width=9.30cm,height=6.20 cm]{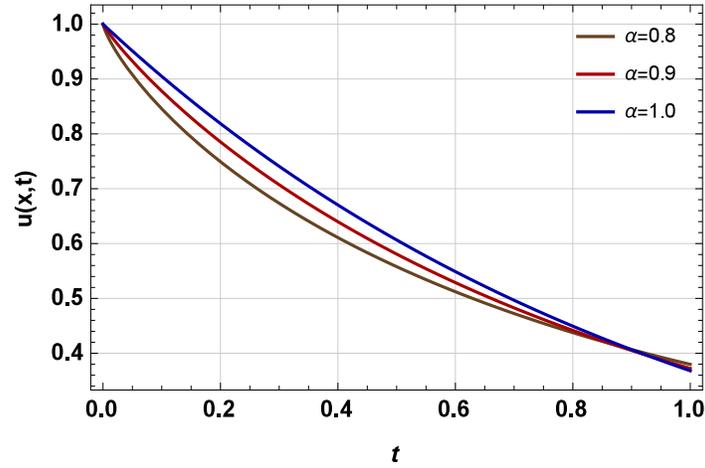}
\caption{Plots of  HPTM solution $u(x,t)$ of Example \ref{ex3} for$\alpha=0.8, 0.9, 1.0$ $t(0,1);x=1$} \label{fig3.11}
\end{figure}
\end{example}
        	
\section{Conclusion}\label{sec-conclu}
 In this paper, \texttt{homotopy perturbation transform method} is successfully employed for numerical computation of initial valued autonomous system of time-fractional model of TFPDE with proportional delay, where we use the fractional derivative in Caputo sense. Three test problems are carried out in order to validate and illustrate the efficiency of the method. The proposed solutions agreed excellently with HPM \cite{SUE16} and DTM \cite{AG11}. The proposed approximate series solutions are obtained without any discretization, perturbation, or restrictive conditions, which converges very fast. However, the performed calculations show that the described method needs a small size of computation in comparison to HPM \cite{SUE16} and DTM \cite{AG11}. Small size of computation of this scheme is the strength of the scheme.

\section*{Acknowledgment}
\noindent The authors are grateful to the anonymous referees for their time, effort, and extensive comment(s) which improve the quality of the paper.
Pramod Kumar is thankful to Babasaheb Bhimrao Ambedkar University, Lucknow, INDIA for financial assistance to carry out the work.

\end{document}